\documentclass[12pt]{amsart}

\usepackage{amssymb, amsfonts,amsthm,latexsym, extpfeil}
\usepackage{eqlist, color}
\usepackage{mathrsfs, textcomp}
\usepackage{graphicx}
\usepackage{bm}
\usepackage{enumitem}

\usepackage{enumitem}

\usepackage[all]{xy}

\makeatletter
\@namedef{subjclassname@2020}{%
  \textup{2020} Mathematics Subject Classification}
\makeatother

\newtheorem{theorem}{Theorem}[section]

\theoremstyle{definition}

\newtheorem{example}[theorem]{Example}
\newtheorem{remarks}[theorem]{Remarks}

\newtheorem{blank}[theorem]{}

\DeclareMathOperator{\coz}{coz}
\DeclareMathOperator{\clop}{clop}

\DeclareMathOperator{\loc}{loc}
\DeclareMathOperator{\Loc}{Loc}
\DeclareMathOperator{\band}{band}
\DeclareMathOperator{\Min}{Min}
\DeclareMathOperator{\BA}{BA}
\DeclareMathOperator{\St}{St}
\DeclareMathOperator{\RC}{RC}
\DeclareMathOperator{\cl}{cl}
\DeclareMathOperator{\interior}{int}

\begin{document}

\baselineskip=17pt


\title[SMP in $\bf{W}$]{Sufficiently many projections in archimedean vector lattices with weak unit}

\author[A. W. Hager]{Anthony W. Hager}
\address{Department of Mathematics \\ Wesleyan University\\
Middletown, CT 06459\\
USA}
\email{ahager@wesleyan.edu}

\author[B. Wynne]{Brian Wynne$^{\ast}$}
\address{Department of Mathematics \\
 Lehman College, CUNY \\
 Gillet Hall, Room 211 \\
 250 Bedford Park Blvd West \\
 Bronx, NY 10468 \\
 USA} \thanks{$^{\ast}$corresponding author}
\email{brian.wynne@lehman.cuny.edu}

\date{}

\begin{abstract}
The property of a vector lattice of sufficiently many projections (SMP) is informed by restricting attention to archimedean $A$ with a distinguished weak order unit $u$ (the class, or category, $\bf{W}$), where the Yosida representation $A \leq D(Y(A,u))$ is available. Here, $A$ SMP is equivalent to $Y(A,u)$ having a $\pi$-base of clopen sets of a certain type called ``local". If the unit is strong, all clopen sets are local and $A$ is SMP if and only if $Y(A,u)$ has clopen $\pi$-base, a property we call $\pi$-zero-dimensional ($\pi$ZD). The paper is in two parts: the first explicates the similarities of SMP and $\pi$ZD; the second consists of examples, including $\pi$ZD but not SMP,  and constructions of many SMP's which seem scarce in the literature.
\end{abstract}



\maketitle

\section{Introduction/preliminaries}\label{Sec1}

\begin{blank}
The property SMP of a vector lattice $A$ is that each non-$\{ 0 \}$ band contains a non-$\{ 0 \}$ projection band (definitions provided below). This appears to have been introduced in \cite[Definition 30.3]{LZ71}; a literature has developed (see \cite{AB78}), which to our knowledge, contains few examples without the principal projection property PPP (every principal band is a projection band). We shall provide many here, from several constructions, and as outlined in the Abstract, after first examining the relation for $(A,u) \in \bf{W}$ of  the vector lattice property SMP to purely topological properties of the Yosida space $Y(A,u)$ (details of the Yosida representation provided below). 

We shall proceed to outlining material on bands, etc., the Yosida representation in $\bf{W}$, $A \leq D(Y(A,u))$, and the interpretation of SMP therein, which has everything to do with clopen sets in $Y(A,u)$.

Most terminology and notations are explained in Sections \ref{Sec1} and \ref{Sec2}. We note especially: ZD and $\pi$ZD, and covers in Section \ref{Sec3}. Local and $\lambda A$ in Section \ref{Sec4}. The constructs $G(X,\tau)$ in Section \ref{Sec5}, and $A(X, \mathcal{Z})$ in Section \ref{Sec6}.

Our main references are \cite{LZ71} (for vector lattices); \cite{AF89}, \cite{BKW77}, \cite{D95} (for $\ell$-groups, more general than vector lattices); \cite{GJ60} (for $C(X)$ and related topology, which theory strongly motivates the study of $\bf{W}$);  \cite{E89} (for topology).
\end{blank}

\begin{blank}
\underline{Bands and SMP}. (from \cite{LZ71})

In a vector lattice $A$:

If $S \subseteq A$, then $S^d = \{ a \in A \mid \vert a \vert \wedge \vert s \vert = 0 \, \forall s \in S \}$. If $S = \{ s \}$, we sometimes write $s^d$ instead of $S^d$.

A \emph{band} is an ideal closed under existing suprema of subsets of its elements. If $S \subseteq A$, then $S^{dd}$ is called \emph{the band generated by $S$ in $A$}.

$A$ has \emph{sufficiently many projections} -- we write SMP -- if each non-zero band contains a non-zero projection band (a band $B$ for which $A = B \oplus B^c$ for some band $B^c \subseteq A$). SMP implies archimedean, which implies that band $=$ polar ($B = B^{dd}$), and the $B^c$ previous is $B^{d}$.

If $A$ has a weak unit (an $e > 0$ with $e^{d} = \{ 0 \}$), then each projection band is principal, so then $A$ has SMP if and only if each non-zero principal band contains a non-zero projection band, thus a projection element.
\end{blank}

\begin{blank}\label{1.3}
\underline{Some topology}. (see \cite{GJ60}, \cite{E89})

(a) All topological spaces $X$ will be at least Tychonoff, usually compact.
\[
D(X) = \{ f \in C(X, \mathbb{R} \cup \{ \pm \infty \}) \mid f^{-1}(\mathbb{R}) \mbox{ dense in } X\}
\]
Here, $\mathbb{R}$ is the reals and $\mathbb{R} \cup \{ \pm \infty \}$ is given the obvious order and topology. This $D(X)$ is a lattice, and closed under multiplication by real scalars; addition and multiplication are only partly defined. We note that $+$ and $\cdot$ are fully defined if $X$ is extremally disconnected  (or more generally if $X$ is quasi-$F$).

For $f \in D(X)$, $\coz(f) = \{ x \in f(x) \neq 0 \}$, and for $S \subseteq D(X)$, $\coz(S) = \{ \coz(s) \mid s \in S \}$.

(b) $U \subseteq X$ is called \emph{clopen} if $U$ is both closed and open, and $\clop(X) = \{ U \subseteq X \mid U \mbox{ is clopen}\}$. This is a Boolean Algebra (BA), a sub-BA of the power set of $X$.

$X$ is called \emph{zero-dimensional} (ZD) if $\clop(X)$ is a base for the topology on $X$.

A \emph{$\pi$-base} for $X$ is a family $\mathcal{B}$ of open sets such that any nonempty open set contains a nonempty $B \in \mathcal{B}$. We say that $X$ is $\pi$ZD if $\clop(X)$ is a $\pi$-base for $X$.

(c) (This material will be mentioned immediately in \ref{1.4}, but not used until Section \ref{Sec5}.) A continuous surjection of compact spaces $X \xtwoheadleftarrow{\tau} Y$ is called \emph{irreducible} if $F$ closed in $Y$, $F \neq Y$ implies $\tau(F) \neq X$. Then $(Y,\tau)$ (or $Y$, or $\tau$) is called a \emph{cover} of $X$. Any $X$ has the maximum cover, the Gleason cover, which we denote $gX$; it is the unique extremally disconnected cover. Generally, an $X$ has many covers, arising from topological and algebraic considerations. References are \cite{H89} and \cite{PW88} (our recent \cite{HW20}, \cite{HW24b}, and \cite{CH} and their many references).
\end{blank}

\begin{blank}\label{1.4}
\underline{In \bf{W}}. (See \cite{HR77})

(a) $\bf{W}$ is the category with objects the archimedean vector lattices $A$ with distinguished weak order unit $u_A$, morphisms $(A,u_A) \xrightarrow{\varphi} (B, u_B)$ the vector lattice homomorphisms with $\varphi(u_A) = u_B$. Now suppressing $u_A$; if $A \in \bf{W}$, the Yosida space $YA$ is the compact Hausdorff space whose points are ideals $M$ maximal for $u_A \notin M$, with the hull-kernel topology. Then, there is the Yosida representation $A \approx \eta(A) \subseteq D(YA)$, for which $\eta(u_A) = 1$, the constant function, and $\eta(A)$ 0-1 separates closed sets in $YA$, so $\coz(\eta(A))$ is an open base for the space $YA$. As noted earlier, addition may be only partially defined in $D(YA)$, but is fully defined in $\eta(A)$.

For $A \in \bf{W}$, we suppress notation further writing $A \leq D(YA)$, $\coz(A)$, etc.

If $U \in \clop(YA)$, then its characteristic function $\chi(U) \in A$ (since $A$ 0-1 separates $U$ and $YA-U$).

We write $A^*$ for the set of all $a \in A$ such that $\vert a \vert \leq nu_A$ for some $n \in \mathbb{N}$ (the natural numbers). In particular, we write $C^*(X)$ for the bounded functions in $C(X)$.

(b) A $\bf{W}$-morphism $A \xrightarrow{\varphi} B$ has its ``Yosida dual" $YA \xleftarrow{Y\varphi} YB$, and $\varphi$ is injective if and only if $Y\varphi$ is surjective. Further, $\varphi$ is injective (an embedding) which is essential in $\bf{W}$ (categorically; see \cite{HS79}) if and only if $Y\varphi$ is a cover (irreducible). This realizes $D(gYA)$ as the maximum essential extension of $A$.

We have explicated the meaning of $A \in \bf{W}$, its $YA$, and the Yosida representation $A \leq D(YA)$. The rest of this paper is in that setting, which will occasionally be repeated.
\end{blank}

\begin{blank}\label{1.5}
\underline{Bands in $A$ and clopen sets in $YA$}.

We continue with $A \in \bf{W}$, as $A \leq D(YA)$.

The following discussion is drawn from \cite{HM96}, \cite{HKM03}, and \cite{HW24}.

For $a \in A$, $\coz(a) \subseteq YA$; for $S \subseteq A$ and $\coz(S) = \bigcup \{ \coz(s) \mid s \in S \}$.

Then the band generated by $S$ is $\{ f \in A \mid \coz(f) \subseteq \overline{\coz(S)} \}$ (the closure operation being that of $YA$). In particular, the \emph{principal band} $a^{dd} = \{ f \in A \mid \coz(f) \subseteq \overline{\coz(a)} \}$.

Since $\coz(A)$ is an open base for $YA$, the $\coz(S)$ (respectively, $\overline{\coz(S)}$) are just the general open sets (respectively, regular closed sets), so the bands are faithfully parametrized by the regular closed $E$, as $\band(E) = \{ f \in A \mid \coz(f) \subseteq E \}$.

If such $\band(E)$ is a projection band (thus principal), then $E \in \clop(YA)$ (emphatically not conversely, as will be seen), and $\band(E) = \chi(E)^{dd}$.

Summing up, and changing the notation: Every projection band is of the form, for $U \in \clop(YA)$, $\band(U) = \chi(U)^{dd}$, and the decomposition $A = \chi(U)^{dd} \oplus \chi(U)^d$ is, for $f \in A$, $f = f\chi(U) + f\chi(U^c)$ ($U^c = YA - U$). That is, these $U$ satisfy: $f \in A$ implies $f\chi(U) \in A$ (n.b. $D(YA)$ and $A$ have only partial multiplication, but these $U$ have all $f\chi(U) \in A$.)

We call these $U$ ``of local type" (with respect to $A$), the terminology to be explained shortly.
\end{blank}


\section{SMP in $\bf{W}$}\label{Sec2}

This section is a reprise of Section 5 of \cite{HW24b}, with additional comments. Always $A \in \bf{W}$, $YA$ may be denoted $Y$, and $A \leq D(YA)$.

\begin{blank}\label{2.1}
\underline{Local type}.

(a) Definition. $U \in \clop(YA)$ is of \emph{local type} if $a\chi(U) \in A$ for every $a \in A$. Let $\lambda A$ be the family of all such $U$.

(b) $\lambda A$ is a sub-BA of $\clop(YA)$ (is easily shown).
\end{blank}

The following explains the terminology.

\begin{blank}\label{2.2}
\underline{Local}. (\cite[Section 5]{HR78}) Suppose $(A,u) \in \bf{W}$, $A \leq D(Y)$.

(a) $f \in D(Y)$ is \emph{locally in $A$} if for every $ x \in Y$, there are $U_x$ open in $Y$ and $a_x \in A$ such that $x \in U$ and $f(y) = a_x(y)$ for every $y \in U_x$. Since $Y$ is compact, finitely many $U_x$ and $a_x$ suffice, and if $Y$ is ZD, the $U_x$ can be clopen. Let $\loc A$ consist of all $f \in D(Y)$ with $f$ locally in $A$. Then, $(\loc A, u) \in \bf{W}$ and $Y(\loc A,u) = Y$.

Call $A$ \emph{local} if $A = \loc A$, and let $\Loc$ be the collection of all $(A,u) \in \bf{W}$ such that $A$ is local. The operator $\loc \colon \bf{W} \to \Loc$ is a monoreflection. (See \cite{HS79} for background on reflections and monoreflections.)

If $(A,u)$ has $u$ a strong unit ($A^* = A$) or $u$ the identity for an $f$-ring multiplication on $A$, then $A$ is local.

(b) (\cite[Remarks 2.3(d)]{HM96}, also \cite{HKM03}) If $A$ is local, then $a\chi(U) \in A$ whenever $U \in \clop(Y)$ and $a \in A$. The converse holds if $Y$ is ZD.

(c) If $A$ is not local, then $A$ is not PPP (\cite[Theorem 2.2]{HM96}).
\end{blank}

We quote from \cite[Section 5]{HW24b}.

\begin{theorem}\label{2.3}
\begin{itemize}
\item[(a)] $A$ has SMP if and only if $\lambda A$ is a $\pi$-base in $YA$.
\item[(b)] If $A$ has SMP, then $YA$ is $\pi$ZD.
\item[(c)] If $YA$ is $\pi$ZD and $\lambda A = \clop(YA)$, then $A$ has SMP.
\item[(d)] These are equivalent: $\loc A$ is SMP; $A^*$ is SMP; $YA$ is $\pi$ZD.
\end{itemize}
\end{theorem}

Thus, given $A$, recognizing that $A$ has or has not SMP, requires some recognition of the $U \in \lambda A$. This will be further highlighted in the following Section \ref{Sec3}, and seen in our various examples in Sections \ref{Sec4} - \ref{Sec8}.

We note: if $YA$ is connected, then $\clop(YA) = \lambda A = \{ \emptyset, YA \}$, so the only projection bands are $\{ 0 \}$ and $A$, but this last can occur even if $YA$ is ZD (see Section \ref{Sec8} below).

Another view of an algebraic significance of $\pi$ZD appears in \cite{LRM21}.


\section{Further similarities of SMP and $\pi$ZD}\label{Sec3}





In \ref{3.1} below we give three ways of describing each of $\pi$ZD and SMP, and include the very easy proof.

After that, we shall only sketch: In \ref{3.2}, some interpretations of \ref{3.1} involving Boolean Algebras, Stone and Yosida spaces, and covers; in \ref{3.3}, some variants of the conditions in \ref{3.1}. These \ref{3.2},\ref{3.3} represent food for thought. Most definitions, etc., are in some sense familiar and are omitted; and we omit proofs, which with details, can be lengthy, and the paper is long enough.

\begin{theorem}\label{3.1}
($\pi$ZD) For compact $K$, the following are equivalent.
\begin{itemize}
\item[(1)] $K$ is $\pi$ZD.
\item[(2)] For each $W$ open in $K$, there is $\mathcal{U} \subseteq \clop(K)$ with $\overline{\bigcup \mathcal{U}} = \overline{W}$.
\item[(3)] For each $C \in \coz(C(K))$, there is  $\mathcal{U} \subseteq \clop(K)$ with $\overline{\bigcup \mathcal{U}} = \overline{C}$.
\end{itemize}
(SMP) For $A \in \bf{W}$, with its $YA$ and $A \leq D(YA)$, the following are equivalent.
\begin{itemize}
\item[(1)] $A$ has SMP.
\item[(2)] For each $W$ open in $YA$, there is $\mathcal{U} \subseteq \lambda A$ with $\overline{\bigcup \mathcal{U}} = \overline{W}$.
\item[(3)] For each $C \in \coz(A)$, there is $\mathcal{U} \subseteq \lambda A$ with $\overline{\bigcup \mathcal{U}} = \overline{C}$.
\end{itemize}
\end{theorem}

\begin{proof}
For $\pi$ZD only. For SMP, just replace $\clop(K)$ by $\lambda A$.

(1) $\Rightarrow$ (2). Let $X$ be the closure of $\bigcup \{ U \mid \clop(K) \ni U \subseteq W \}$. So $X \subseteq \overline{W}$. If $X \neq \overline{W}$, there is open $V \neq \emptyset$ with $V \cap X = \emptyset$ and $V \cap W \neq \emptyset$, but also there is $U \in \clop(K)$ with $\emptyset \neq U \subseteq V \cap W$ by (1). Contradiction.

(2) $\Rightarrow$ (3). Each $C$ is a $W$.

(3) $\Rightarrow$ (1). The $C$'s form a base, so if $W \neq \emptyset$, there is $C \neq \emptyset$ with $\overline{C} \subseteq W$. Apply (3) to get $\overline{\bigcup \mathcal{U}} = \overline{C}$. So there is $U \neq \emptyset$ with $U \subseteq W$. 
\end{proof}

\begin{blank}\label{3.2}
\underline{Some interpretations of \ref{3.1}}.

We begin with Boolean Algebras (BA). This is natural since $A$ has SMP if and only if the BA of projection bands is order dense in the BA of all bands. (These are BA's by \cite[13.7]{D95} and \cite[30.2]{LZ71}.)  Then, through Stone and Yosida duality, move into topology.

(a) For $\mathscr{A} \in \BA$, $\St(\mathscr{A})$ will denote the Stone space of $\mathscr{A}$, whereby $\mathscr{A}$ is $\BA$-isomoprhic to $\clop(\St(\mathscr{A}))$.

For compact $K$, $F(K)$ denotes $\{ f \in C(K) \mid \vert f(K) \vert < \omega \}$. Then, $F(K) \in \bf{W}$, $YF(K) = \St(\clop(K))$, and this is also the ZD-reflection of $K$, denoted $zK$. (See \cite{Herr68} about reflections, and also \cite{HS79} more generally.)

Let $\RC(K)$ denote the BA of regular closed subsets of $K$ (those $E$ with $E = \cl \interior E$). Then, $\St(\RC(K)) = gK$, the Gleason cover, which is the maximum cover.

For BA's, $\mathscr{A}$ is order-dense in $\mathscr{B}$ if and only if the Stone dual map $\St(\mathscr{A}) \leftarrow \St(\mathscr{B})$ is a cover. (This is well-known, follows from the corresponding fact in $\bf{W}$ applied to $F(\St(\mathscr{A})) \leq F(\St(\mathscr{B}))$.)

One may see \cite{S69}, \cite{PW88} (and many other references) for most of the above. The assertion about $F(K)$ and $zK$ may be novel, but are quite easy.

The following now becomes obvious. \\

(b) \underline{Theorem}. For compact $K$ (respectively, $A \in \bf{W}$), the following are equivalent.
\begin{itemize}
\item[(i)] $K$ is $\pi$ZD (respectively, $A$ has SMP).
\item[(ii)] In $\BA$'s, $\clop(K)$ is order-dense in $\RC(K)$ (respectively, $\lambda A$ is order-dense in $\RC(YA)$).
\item[(iii)] $zK \twoheadleftarrow K$ (respectively, $\St(\lambda A) \twoheadleftarrow YA$) is a cover. 
\end{itemize} 
\end{blank}

Understanding the meaning of the SMP-part of (iii) depends on understanding $\lambda A$. This is not so easy, as our many examples in Sections \ref{Sec4} - \ref{Sec8} show.

\begin{blank}\label{3.3}
\underline{Variations of the conditions}.

We offer just a few.

Each condition in the two parts of \ref{3.1} has an assumption and a conclusion about $W/C$ and $\mathcal{U}$. We code variations as illustrated:

[$\pi$ZD(2) conclude $\bigcup \mathcal{U} = W$] means change \ref{3.1}$\pi$ZD(2) to: for all open $W$ there exists $\mathcal{U} \subseteq \clop(K)$ with $\bigcup \mathcal{U} = W$.

[$\pi$ZD(2) assume $\vert \mathcal{U} \vert < \omega_1$] means change \ref{3.1}$\pi$ZD(2) to: for all open $W$ there exists $\mathcal{U} \subseteq \clop(K)$ with $\vert \mathcal{U} \vert < \omega_1$ and $\overline{\bigcup \mathcal{U} }= \overline{W}$.

Having understood that, we note:
\begin{itemize} 
\item $\pi$ZD(2) conclude $\bigcup \mathcal{U} = W$.
This is equivalent to: $K$ is ZD.

\item SMP(2) conclude $\bigcup \mathcal{U} = W$.
This is equivalent to: $A$ has SMP and $YA$ is ZD.

\item $\pi$ZD(2) assume $\vert \mathcal{U} \vert < \omega_1$.
This is equivalent to: $K$ is ZD and each regular closed set is the closure of a cozero-set. This latter condition is called ``$K$ is fraction-dense" in \cite{HM93}, where appears an interesting and stubborn unsolved problem: Does fraction-dense imply strongly fraction dense? See \cite{HM93} for definitions and details.

\item $\pi$ZD(3) assume $\vert \mathcal{U} \vert< \omega_1$.
This is equivalent to [SMP(3) assume $\vert \mathcal{U} \vert < \omega_1$]. Thus the latter is, for $A \in \bf{W}$, just a topological property of $YA$ - call it $\mathscr{P}$ - which has ZD $\Rightarrow$ $\mathscr{P}$ $\Rightarrow$ $\pi$ZD.

Neither implication reverses: ZD $\not \Leftarrow \mathscr{P}$ is shown by $C(K)$ for $K$ a compactification of $\mathbb{N}$ with remainder $[0,1]$ (see \cite[Example 3.5.15]{E89}). And $\mathscr{P} \not \Leftarrow \pi$ZD is shown by $C(K)$ for $K$ Alexandroff's  double circumference. 

And, ($\mathscr{P}$ for $YA$) $\Leftrightarrow$ Each principal band of $A$ is generated (qua bands) by $< \omega_1$ projection elements of $A$. This property of $A$ is called ``super SMP" in \cite{L86}, where the second example above appears.
\end{itemize}
\end{blank}

Among other variations possible, one can replace the $\omega_1$'s above by any cardinal $\geq \omega_1$. We do not go there. 

We return to the main narrative.


\section{A very simple class of SMP's}\label{Sec4}

\begin{theorem}\label{4.1}
Suppose $A \in \bf{W}$.
\begin{itemize}
\item[(a)] If $p$ is an isolated point of $YA$, then $\{ p \} \in \lambda A$.
\item[(b)] If the set of isolated points of $YA$ is dense, then $A$ has SMP.
\end{itemize}
\end{theorem}

\begin{proof}
(a) Here, $\{ p \} \in \clop(YA)$, so $\chi(\{ p \}) \in A$, so for every $r \in \mathbb{R}$, any $r\chi(\{ p \}) \in A$ and $a \chi(\{ p \}) = a(p)\chi(\{ p \})$.

(b) Here, $\{ \{ p \} \mid p \mbox{ isolated} \}$ is a $\pi$-base. Apply \ref{2.3}(a).
\end{proof}

Here is an interesting instance of \ref{4.1}, illustrating some of the information in \ref{2.3}.

\begin{example}
(from \cite{HM96})

Define $A \leq C(\mathbb{N} \times \{ 0 , 1\}) = D(\beta \mathbb{N} \times \{ 0, 1\})$ as: $a \in A$ means $a(n,0) - a(n,1)$ is a bounded function of $n$. Here, $YA = \beta(\mathbb{N} \times \{ 0 ,1 \})$, in which the set of isolated points is $\mathbb{N} \times \{ 0 ,1 \}$. So \ref{4.1} says $A$ has SMP. The following can be shown easily.
\begin{itemize}
\item $\lambda A$ is the finite/cofinite Boolean algebra.
\item $A^* = C^*(\mathbb{N} \times \{ 0, 1 \})$.
\item $\loc A = C(\mathbb{N} \times \{ 0 , 1 \})$.
\end{itemize}
Note that $A$ is not PPP because $A$ is not local.
\end{example}


\section{Some SMP extensions of $A$; another simple class of SMP's}\label{Sec5}

\begin{blank}\label{5.1}
\underline{The method}.

We deploy a generalization of ideas in \cite{HM96}, \cite{HKM03}, \cite{HW24} (and earlier even; see these papers).

Take an $A \in \bf{W}$, and a cover $YA \xtwoheadleftarrow{\tau} X$ with $X$ ZD. (There are many of these noted in \ref{5.3}(c) below.) For $a \in A$, $\tilde{a} = a \circ \tau \in D(X)$, since $\tau$ is a cover. Now take $U \in \clop(X)$; then $\tilde{a}\chi(U) \in D(X)$. For two of these
\[
\tilde{a}_1\chi(U_1) + \tilde{a}_2\chi(U_2) = \tilde{a}_1\chi(U_1 - U_2) + \widetilde{(a_1+a_2)}\chi(U_1 \cap U_2) + \tilde{a}_2\chi(U_2 - U_1).
\]
Let $G = G(A,\tau)$ be the group generated by all these $\tilde{a}\chi(U)$ in $D(X)$.

Observe that any $g \in D(X)$ has $g \in G$ if and only if $g$ is a finite sum $\sum_{i=1}^n \tilde{a_i}\chi(U_i)$ with the $U_i$ disjoint (as with the sum of two, above).
\end{blank}

We now have the following.

\begin{theorem}\label{5.2}
The $G(A, \tau) \equiv G$ just defined is a vector lattice, thus $G \in \bf{W}$ with $YG = X$. And, $\lambda G = \clop(YG)$, so $G$ is local and SMP.
\end{theorem}

\begin{proof}
$G$ is a vector lattice: It suffices that $g \in G$ implies $g \vee 0 \in G$ (\cite[Theorem 6.1]{D95}). 

First, take $g = (a\circ \tau)\chi(U)$. Then $g \vee 0 = [(a \vee 0) \circ \tau]\chi(U) \in G$ since $a \vee 0 \in A$.

Second, if $g$ is a finite sum $\sum_{i=1}^n a_i\chi(U_i)$ with the $U_i$ disjoint, then $g \vee 0 = \sum_{i=1}^n [\tilde{a}_i\chi(U_i) \vee 0]$ and each $\tilde{a}_i\chi(U) \vee 0  \in A$ by the above.

So $G \in \bf{W}$ and $YG = X$ (most simply, since all $\chi(U) \in G$ and these separate points of $X$ by ZD).

Now $\lambda G \equiv \{ U \in \clop(X) \mid g\chi(U) \in G \, \, \forall g \in G \}$. We have $g = \sum_{i =1}^n (a_i \circ \tau)\chi(U_i)$, where the terms are pairwise disjoint, and $g\chi(U) = \sum_{i=1}^n(a_i \circ \tau)\chi(U_i \cap U) \in G$ for every $U \in \clop(YG)$. So $\lambda G = \clop(YG) = \clop(X)$, and since $X$ is ZD, $G$ is local and SMP by \ref{2.3}.
\end{proof}

\begin{blank}\label{5.3}
\underline{Particular cases of $G(A,\tau)$}.

(a) If $A = A^*$, then in constructing $G(A,\tau)$, we could use any continuous surjection $YA \xtwoheadleftarrow{\tau} X$ with $X$ ZD. This makes $G(A,\tau)$ a case of \ref{2.3}(b) and (c), and will be ignored.

(b) If $YA$ is already ZD, we could use $YA \xleftarrow{\tau} YA$ the identity. This makes $G(A,\tau) = \loc A$ (the local reflection), and ``exhibits" \ref{2.3}(d).

(c) For general $A \in \bf{W}$, there are many ZD covers $YA \xtwoheadleftarrow{\tau} X$, thus many essential extensions to SMP-objects $A \leq G(A, \tau)$. The projectable (PPP, per \cite[Section 2]{AB78}; also, \cite[Section 25]{LZ71}) hull $A \leq pA$ is such a situation $YA \xtwoheadleftarrow{\tau} YpA$, and then $G(A, \tau) = pA$ as shown in \cite{HKM03}. This is generalized to more examples using other hulls in \cite{HW24}.
\end{blank}


\section{A lemma toward more examples}\label{Sec6}

The following will be used for subsequent examples; the purpose is to reduce computations.

Let $X$ be any compact space, and $\mathcal{Z}$ a disjoint family of nonempty nowhere dense zero-sets in $X$ ($\mathcal{Z}$'s of various nature will be used later). For each $Z \in \mathcal{Z}$, let $f_Z \in D(X)$ have $f_Z \geq 1$ and $f_Z^{-1}(+\infty) = Z$. Let $S$ be the linear span of $\{ f_Z \mid Z \in \mathcal{Z} \}$, i.e., $s \in S$ means $s= \sum_{i \in I} r_if_{Z_i}$ ($r_i \in \mathbb{R}$) for some finite $\{ Z_i \}_{i \in I} \subseteq \mathcal{Z}$. (Note. In such a sum, we can assume the $Z$'s are all distinct and therefore disjoint by collecting and summing the $f_Z$'s with the same $Z$, and adjusting the $r$'s (all $\neq 0$).)

It is convenient to write such $s$ as $s = \sum_{i \in I} r_i f_i$ with $f_i$ standing for $f_{Z_i}$. 

Now let $A = A(X, \mathcal{Z}) \equiv C(X) + S$. One sees that $A$ is a group in $D(X)$ that is closed under multiplication by real scalars (e.g., check that $g + f_Z \in D(X)$ whever $g \in C(X)$). Our point is this.

\begin{theorem}\label{6.1}
$A$ is a vector lattice. Thus, $A \in \bf{W}$ with $YA = X$.
\end{theorem}

\begin{proof}
$a \in A$ means $a = g+s$ with $s \in S$ and $g \in C(X)$, and equivalently, $a \in D(X)$, with an $s \in S$ such that $a - s$ is bounded. 

It suffices to show that each $a \in A$, $a \vee 0 \in A$, according to \cite[Theorem 6.1]{D95}.

Let $a = g + \sum_{i \in I} r_i f_i$ with $g \in C(X)$, each $r_i \neq 0$, and the $Z_i = f_i^{-1}(+\infty)$ disjoint (as discussed above).
Let $I = P \cup N$, where $P = \{ i \in I \mid r_i > 0 \}$ and $N = \{ i \in I \mid r_i < 0 \}$. 
Since $X$ is compact Hausdorff, and hence normal, there are disjoint open $\{ U_i \}_{i \in I}$ such that $U_i \supseteq Z_i$. By shrinking $U_i$ as needed, one may assume that $a > 0$ on $U_i$ when $i \in P$ and $a < 0$ on $U_i$ when $i \in N$ (e.g., if $i \in P$, then for each $x \in Z_i$, one may can choose an open $W_x$ such that $x \in W_x$ and $a > 0$ on $W_x$, and then redefine $U_i$ to be $(\bigcup_{x \in Z_i} W_x ) \cap U_i$).  Let $U = \bigcup_{i \in I} U_i$ and $U' = X-U$. 

Let $V = X - (\bigcup_{i \in I} Z_i)$. Then $\{ V \} \cup \{ U_i \}_{i \in I}$ is an open cover of $X$. Since $X$ is compact, and hence normal, there is a partition of unity $\{ \rho_V \} \cup \{ \rho_i \}_{i \in I} \subseteq C(X)$  dominated by that cover, i.e., $\overline{\coz(\rho_V)} \subseteq V$, $\overline{\coz(\rho_i)} \subseteq U_i$ for $i \in I$, and $1 = \rho_V + \sum_{i \in I} \rho_i$ (see, e.g., \cite[Theorem 36.1]{Munkres}). 
Note that $r_jf_j \rho_i \in C(X)$ whenever $j \neq i$; and that $r_if_i \rho_V \in C(X)$.

We show that $a \vee 0 \in A$ if $\sum_{i \in I} (a \vee 0 )\rho_i \in A$. Now 
\[
a \vee 0 = (a \vee 0)\rho_V + \sum_{i \in I}(a \vee 0)\rho_i \hspace{.2in} \mbox{(pointwise in $D(X)$)} 
\]
From the choice of $V$, we know $a$ is finite-valued on $V$, so too $a \vee 0$. Since $\rho_V \in C(X) \leq A$ and $\overline{\coz(\rho_V)} \subseteq V$, it follows that $(a\vee 0)\rho_V \in C(X) \leq A$. Therefore, if $\sum_{i \in I}(a \vee 0)\rho_i \in A$, so is $a \vee 0$.

By the choice of $U_i$, on $\coz(\rho_i)$ we have $(a \vee 0)\rho_i = a\rho_i$ or $(a \vee 0)\rho_i = 0$ depending on whether $i \in P$ or $i \in N$, respectively.
Thus:
\begin{align*}
\sum_{i \in I}(a \vee 0)\rho_i &= \sum_{i \in N} (a \vee 0)\rho_i + \sum_{i \in P} (a \vee 0)\rho_i =\sum_{i \in P} a\rho_i \\
\end{align*}
Now
\[
\sum_{i \in P} a\rho_i = \sum_{i \in P} (g + \sum_{j \in I} r_jf_j)\rho_i = \sum_{i \in P} g\rho_i + \sum_{i \in P} \sum_{j \neq i} r_jf_j\rho_i +\sum_{i \in P} r_if_i\rho_i  
\]
So, using that 
\[
r_if_i = r_if_i \left (\rho_V + \sum_{i \in I} \rho_i\right )= r_if_i\rho_V + \sum_{i \in I} r_if_i\rho_i
\]
one concludes:
\begin{align*}
\sum_{i \in I}(a \vee 0)\rho_i  &= \sum_{i \in P} g\rho_i +\sum_{i \in P} \sum_{j \neq i} r_jf_j\rho_i + \sum_{i \in P} \left [ r_if_i - r_if_i \rho_V \right ]     \\
&= \underbrace{\sum_{i \in P} g\rho_i +\sum_{i \in P}\sum_{j \neq i} r_jf_j\rho_i - \sum_{i \in P} r_if_i\rho_V}_{\in C(X)}+ \underbrace{\sum_{i \in P} r_if_i}_{\in S} 
\end{align*}
Hence $\sum_{i \in I}(a \vee 0)\rho_i \in A$ as desired.
\end{proof}


\section{An example of $\pi$ZD not SMP}\label{Sec7}

We have an $A \in \bf{W}$ with these features; in fact, $YA$ will be ZD. This $A$ will be of the form $A(X,\mathcal{Z})$, per Section \ref{Sec6}.

Let $X$ be the Cantor set. (We note below general features of $X$ which serve this work.) Take any clopen $Y \neq \emptyset, X$, let $Z = X - Y$, and take $\{ y_n \}_{n \in \mathbb{N}}$ and $\{ z_n \}_{n \in \mathbb{N}}$ countable dense sets of $Y$ and $Z$, respectively. Our $\mathcal{Z}$ will be $\{ \{y_n, z_n \} \mid n \in \mathbb{N} \}$. So, for each $n$, take $f_n \in D(X)$ with $f_n \geq 1$ and $f_n^{-1}(+\infty) = \{ y_n, z_n \}$, then $S$ is all finite sums $\sum_{i \in I} r_i f_i$ ($r_i \in \mathbb{R}$), and $A(X,\mathcal{Z}) \equiv C(X) + S$.

\begin{theorem}\label{7.1}
(examples) $A = A(X, \mathcal{Z}) \in \bf{W}$ has $YA = X$, thus ZD, and fails SMP.
\end{theorem}

\begin{proof}
Using \ref{6.1}, it suffices to show ``fails SMP". We show that the band $\band(Y) \equiv \{ a \in A \mid \coz(a) \subseteq Y \}$ (see \ref{1.5}) contains no nonzero projection band.

First, $\band(Y) \subseteq C(X)$: if $a = g + \sum_{i \in I} r_i f_i$ ($g \in C(X)$) has $\coz(a) \subseteq Y$, then $\coz(f_i) \subseteq Y$ (which contradicts the choice of $f_i$), and therefore $r_i=0$, for each $i \in I$; so $a = g$.

Now, suppose $C \subseteq \band(Y)$ is a nonzero projection band. So $C = \chi(U)^{dd}$ for some $U \in \clop(X)$ with $\emptyset \neq U \subseteq Y$, and some $y_n \in U$. For that $n$, consider $f_n$, and, since $C$ is a projection band, write $f_n = k_1 + k_2$ with $\coz(k_1) \subseteq U \subseteq Y$ (thus $k_1 \in C(X)$) and $\coz(k_2) \subseteq X - U$. Then $k_2(z_n) =f_n(z_n) = +\infty$, which implies $k_2(y_n) = +\infty$, which contradicts $\coz(k_2) \subseteq X - U$. 
\end{proof}

The argument above is valid using any compact $X$ which is $\pi$ZD, has a clopen set $Y$ with a countable dense set of $G_{\delta}$-points (the $\{ y_n \}_{n \in \mathbb{N}}$), with $X - Y = Z$ containing infinitely many $G_{\delta}$-points (the $\{ z_n \}_{n \in \mathbb{N}})$.


\section{Another class of examples, SMP and not}\label{Sec8}

We continue in the setting of Section \ref{Sec6}, with, as there, $A = A(X, \mathcal{Z}) = C(X) + S$. We exhibit, first (\ref{8.1}) some of these that have SMP, and several (\ref{8.3}) some of which have $X$ ZD but $\lambda A = \{ \emptyset, YA \}$, thus failing SMP (as strongly as possible; cf. \ref{7.1}).

\begin{theorem}\label{8.1}
If $X$ is $\pi$ZD and $\bigcup \mathcal{Z}$ is nowhere dense (e.g., $\mathcal{Z}$ is finite), then $A(X, \mathcal{Z})$ has SMP.
\end{theorem}

\begin{proof}
Given open $W \neq \emptyset$, $W - \overline{\bigcup \mathcal{Z}}$ is open, nonempty (since $\bigcup \mathcal{Z}$ is nowhere dense), and then contains clopen $U \neq \emptyset$ (since $X$ is $\pi$ZD), which has all $f_Z$'s bounded on $U$, and therefore $a\chi(U) \in A$ for all $a \in A$.

By \ref{2.3}, $A(X,\mathcal{Z})$ has SMP.
\end{proof}

Since ``local" is tied up with SMP, we note the following. This shows the construct in \ref{8.1} need not be local (hence not PPP).

\begin{theorem}\label{8.2}
Suppose $X$ is ZD. $A(X, \mathcal{Z})$ is local if and only if $\vert Z \vert = 1$ for every $Z \in \mathcal{Z}$. 
\end{theorem}

\begin{proof}
If some $\vert Z \vert > 1$, as $x,y \in Z$ with $x \neq y$, then clopen $U$ with $x \in U$ and $y \notin U$ will have $a\chi(U) \notin A$.

For convenience, write a faithful indexing $\mathcal{Z} = \{ Z_i \}_{i \in I}$, and each $f_{Z_i} = f_i$, and suppose $Z_i = \{ p_i \}$ for each $i \in I$. Given clopen $U$ and $a \in A$, we have $a = g + \sum_{i \in I} r_if_i$ with $g \in C(X)$, $r_i = 0$ for all but finitely many $i \in I$, and 
\[
a\chi(U) = g\chi(U) + \sum\{ r_if_i\chi(U) \mid p_i \notin U \} + \sum \{ r_if_i\chi(U) \mid p_i \in U \}.
\]
The first two terms are bounded, and the third is of the form $s-h$ for some $s \in S$ and some bounded $h$. So $a\chi(U) \in A$, and $A$ is local by \ref{2.2}(b).
\end{proof}

The following is our ``extreme version" of $A \in \bf{W}$ failing SMP with $YA$ ZD. The ``extreme" seems to require complications on top of the general construct in Section \ref{Sec6}.

First, take $A(X,\mathcal{Z}) = C(X) + S$ exactly as in Section \ref{Sec6}. Then, take any cover $X \xtwoheadleftarrow{\tau} Y$ with $Y$ ZD. Any covering map involved preserves nowhere density, so that $\tau^{-1}(Z)$ ($Z \in \mathcal{Z}$) is a nowhere dense zero-set, and thus $\tilde{f}_Z \equiv f_Z \circ \tau \in D(Y)$. ``On $Y$" now, let $\tilde{\mathcal{Z}} \equiv \{ \tau^{-1}(Z) \mid Z \in \mathcal{Z} \}$ and $\tilde{S} \equiv \{ \tau^{-1}(s) \mid s \in S \}$ (this is the linear span of $\{ \tilde{f}_Z \mid Z \in \mathcal{Z} \}$). Now, \ref{6.1} applies to $A(Y,\tilde{\mathcal{Z}}) = C(Y) + \tilde{S}$, which is a $\bf{W}$-object with Yosida space $Y$.

\begin{theorem}\label{8.3}
In the construction above, suppose that
\begin{itemize}
\item $X$ is connected, $\vert X \vert > 1$, and every point is a $G_{\delta}$ (thus a nowhere dense zero-set). E.g., $X = [0,1]$.
\item Suppose $Y$ is ZD. E.g., $Y =gX$ (see \ref{1.3}(c)). 
\end{itemize}
Let $\mathcal{Z} = \{ \{p \} \mid p \in X \}$. Then, $A = A(Y,\tilde{\mathcal{Z}})$ has $YA = Y$, which is ZD, and $\lambda A = \{ \emptyset, Y \}$, thus only the projection bands $\{ 0 \}, A$, and fails SMP (drastically; cf. \ref{7.1}).
\end{theorem}

\begin{proof}
It suffices to show that $\lambda A = \{ \emptyset, Y \}$, i.e., for $U \in \clop(Y)$ with $\emptyset \neq U \neq Y$, there is $a \in A$ with $a\chi(U) \notin A$.

Given such $U$, let $U' = Y - U$. Then, $\tau(U)$ and $\tau(U')$ are closed (since $\tau$ is continuous and $U,U'$ are compact), $X = \tau(U) \cup \tau(U')$ (since $\tau$ is surjective), so there is $p_0 \in \tau(U) \cap \tau(U')$ (since $X$ is connected). With $\{ p_0 \} = Z \in \mathcal{Z}$, let $a = \tilde{f}_Z= f_Z \circ \tau$. There are $y \in U$ and $y' \in U'$ with $\tau(y) = p_0 = \tau(y')$, so $a(y) = +\infty = a(y')$. Let $f = a\chi(U)$. We have $f(y) = +\infty$ and $f(y') = 0$. Towards a contradiction, assume $f \in A$, as $f = h + \tilde{s}$, with $h \in C(X)$. Since $f(y) = +\infty$, it follows that $\tilde{s}(y) = +\infty$, which implies $\tilde{s}(y') = +\infty$, and then $f(y') = +\infty$, a contradiction.
\end{proof}

\begin{remarks}\label{8.4}
We feel it necessary to make, with apologies, the following rather inscrutable comments.

In our paper \cite{HW24b}, we have discussed various interpretations of the Freudenthal Spectral Theorem and the relations to SMP. In 5.5(b) there, we have alluded to Example \ref{8.3} above, which is possibly similar to an example of Veksler \cite[Example 3.3]{V76}. In this connection, it is important to note that \ref{8.3} has the property called ``$\forall e ODe^{dd}$", also ``(WF)$_V$", by virtue of \cite[2.3 and 2.5]{HW24b}. See \cite{HW24b}, and its references, for the meaning of all this.
\end{remarks}

  

\section*{Declarations}

\begin{itemize}
\item Funding: A.W. Hager received no funding; B. Wynne received PSC-CUNY Research Award \#66070-00-54.
\item Conflict of interest: Neither author had any conflicts of interest.
\item Author contribution: Both authors contributed equally to this paper.
\end{itemize}





\end{document}